\documentclass[11pt]{article}
\usepackage{amsmath,amssymb}
\usepackage{amsthm}
\newtheorem{theorem}{Theorem}[section]
\newtheorem{lemma}{Lemma}[section]
\newtheorem{proposition}{Proposition}[section]
\newtheorem{corollary}{Corollary}[section]
\newtheorem{definition}{Definition}[section]
\newtheorem{remark}{Remark}[section]
\numberwithin{equation}{section}

\begin{document}
\title { \bf   The local properties  of the
 Markov processes of Ornstein-Uhlenbeck type \thanks{Supported by
    SRFDP (20060335032) and ZJNSFC (6100176) }}

\author{Jing Zheng$^{1 }$, \qquad Zhengyan Lin$^{2}$,
\qquad Changqing Tong$^{1}$\\
{}  {\small\it 1. Institute of Applied Mathematics, Hangzhou Dianzi
University, Hangzhou,  310018, China}
\\
  {\small\it 2. Department of Mathematics, Zhejiang University,
Hangzhou, 310027, China} }
  \date{}

\maketitle

\vskip3mm
{\noindent}-------------------------------------------------------------------------------------------------------

{\noindent\bf Abstract}\\

     We prove the existence of a local time, the continuity of
     the local time about $t$, and the regular property for $a.e.$ $x\in R$ of
 a  Ornstein-Uhlenbeck type $\{X_t,\ t\in R^+\}$
driven by a general L\'{e}vy process, under mild regularity
conditions. We discuss the asymptotic behaviour of the local time
when $X$ is ergodic. We also investigate the first passage problem.
These results give precise information about the local properties of
the sample functions.


{\noindent}-------------------------------------------------------------------------------------------------------
  \vskip 10mm    \footnotetext[1] { E-mail address:
zhengjing@hdu.edu.cn (Jing Zheng)}

\section{Introduction}

 Let $(\Omega,\mathfrak{F},P)$ be a probability space, and
  $M_+(R^d)$ be the totality of real $d\times d$ matrices whose
all eigenvalues have positive real parts. The starting at $x$ Markov
process of Ornstein-Uhlenbeck (O-U) type   $X=\{X_t,t\in R^+,P_{x}\}
$ over $(\Omega,\mathfrak{F},P)$  is a Feller process
    with
infinitesimal generator
$$
A = G - \sum^d_{j=1}\sum^d_{k=1}Q_{jk}x_k\frac{\partial}{\partial
x_j},
$$
where $G$ is the infinitesimal generator of a L\'{e}vy process
$Z=\{Z_t, t\in R^+\}$ taking values in $R^d$,   $Q\in M_+(R^d)$ and
$x\in R^d$. An equivalent definition of this process $X$ is given by
the unique solution of the equation
$$
X_t= x - \int^t_0QX_sds + Z_t,
$$
which can be expressed as
$$
X_t = e^{-tQ}x + \int^t_0e^{(s-t)Q}dZ_s, \eqno {(1.1)}
$$
where the stochastic integral with respect to the L\'{e}vy process
$Z$ is defined by convergence in probability from integrals of
simple functions.  When $Z$ is the Brownian motion taking values in
$R^d$, $X$ is the ordinary O-U process.

The study of Markov processes of O-U type keeps receiving much
attention both in the physical and the mathematical literature, for
example, in climate models to explain    the so-called
Dansgaard-Oeschger events--see \cite{im:06} and the references
therein. Many authors investigated the recurrence \cite{Sh:85}
\cite{war:98}, the strong Feller property and the exponential
$\beta$-mixing property \cite{ma:04}. In \cite{lin:10}, the authors
used the local time as the kernel function of the empirical
likelihood inference, but as to the author's  knowledge, there is
not paper investigated the existence of the local time of $X$.

In this paper, we will prove the existence of a local time under
 mild regularity conditions, and  its continuous properties about time
 $t$.
  Here we consider the local time as the Radon-Nikodym derivative of
the occupation measure of $X$ relative to a Borel set $A$. There are
some different
 between this definition and the Blumenthal-Getoor local time. We will
 consider their connection.   Many authors like to
 consider $X$ on $R^d$, but when $d\geq 2$, single points are
 essentially polar even for Brownian motion, that is to say, $L(x,t)=0$ for $a.e. $ $x$. Hence,
 we define $X$ on $R$. It is well-known that $X$ is ergodic under very mild
 regularity condition, let $F$ be the unique invariant distribution
 of $X$, if there is a density $f$ of the distribution $F$, then
 $$
\lim_{\epsilon\rightarrow 0}\lim_{t\rightarrow
\infty}\frac{1}{2t\epsilon}\int_0^tI_{\{|X_s-x|<\epsilon\}}ds=f(x).
 $$
 Under some conditions, we have
 $$
 \lim_{t\rightarrow \infty}\frac{L(x,t)}{t}=f(x)\qquad   L^2(P),
 $$
that is to say, $ \frac{L(x,t)}{t}$ or $\
\frac{1}{2t\epsilon}\int_0^tI_{|X_s-x|<\epsilon}ds$ is a unbiased
consistent estimator of $f$, but we do not go any further in this
direction in this paper.

The paper is organized   as follows. we will give some basic results
about the L\'{e}vy process and $X$ in section 2.
  In Section 3, we establish
completely general criteria for the existence of the local time of
$X$ in terms of Fourier analysis following Berman \cite{berman:73}.
In section 4, we shall discuss the continuity and the asymptotic
behavior of the local time about time $t$. In section 5, we study
the first passage cross a lever.

 Throughout  this paper, $x$ is the starting point of the Markov process
  of O-U type.  $C$ always stands for a positive
constant, whose value is irrelevant. The expectation operator under
$P$ and $P_x$ are denoted  by $E$ and $E_x$. $\varphi_\eta(\cdot)$
stands for the characteristic function of a random variable or a
distribution $\eta$ and $\mathfrak{F_t}$ is the $P$-completed
sigma-field generated by $(X_s, s\leq t)$.

\section{Preliminaries}

 In this section, we   collect some  basic results about
the L\'{e}vy  process and the Markov process of O-U type which will
be used in the following section.

 Let $\{Z_t,\ t\in R^+\}$ be a L\'{e}vy process taking values in $R $, whose characteristic
 function is given by
 $$
 E(e^{i   \theta Z_t }) = \exp \{-t\psi(\theta)\},
 $$
 where
 $$
 \psi(\theta) =ib\theta+
 \frac{ \sigma^2\theta^2}{2} -
 \int_{R }(e^{i \theta u }-1-\frac{i\theta u}{1+|u|^2})\rho(du)\eqno{(2.1)}
$$
is called the L\'{e}vy exponent, $b\in R$, $\sigma \geq 0$   and
$\rho$ is a measure on $R $ satisfying that $\rho(\{0\})=0$ and the
integrability condition
$$
\int_{R^d} (1  \wedge |u|^2)\rho(du)< \infty.
$$

Certainly, the process $Z$ is characterized by the generating
triplet $(b,\sigma,\rho(\cdot))$.

Let $X=\{X_t,t\in R^+,P^{x}\} $ be a Markov process of O-U type
defined by    $(1.1)$. The next proposition specifies
 the characteristic function of the transition probability of $X$.

\begin{proposition}
 {\bf (Sato and Yamazato 1984, Theorem 3.1)}

 Let $P(t,x,\cdot)$ be the transition probability of $X$.  The characteristic function
 of $P(t,x,\cdot)$ is
 $$
 \varphi_{P(t,x,\cdot)}(\theta) = \exp \Big\{i x e^{-tQ }\theta -
 \int^t_0\psi(e^{  -sQ }\theta)ds\Big\}, \eqno{(2.2)}
 $$
where $\psi$ is given in $(2.1)$. In particular, the generating
triplet of $P(t,x,\cdot)$ is given by $(b_{t,x},\sigma_t,\rho_t)$,
where
\begin{eqnarray*}
b_{t,x}&=& e^{-tQ}x+\int_0^te^{-sQ} b ds+\int_R\int_0^te^{-sQ}z\{I_{\{|e^{-sQ}z|\leq 1\}}-I_{\{|z|\leq 1\}}\}ds\rho(dz), \\
 \sigma^2_t&=&\int_0^te^{-2sQ}\sigma^2ds,  \\
 \rho_t(E)&=&\int_0^t\rho(e^{sQ}E)ds.
\end{eqnarray*}
\end{proposition}

Now assume that
$$
\int_{|z|>1}\log|z|\rho(dz)<\infty,\eqno{(2.4)}
$$
or, equivalently, $E[\log\{max(1,|Z_1|)\}]<\infty$.

\begin{proposition}{\bf (Sato and Yamazato 1984, Theorem 4.1 and
4.2)}

 (a) If  (2.4) holds, there exists a limit distribution $F$ such that
 $$P(t,x,A)\rightarrow F(A), \quad \mbox{as}\quad t\rightarrow \infty$$
for any $x\in R$ and $A\in\mathfrak{B}(R)$. This $F$ is the unique
invariant distribution of $X$. Moreover, the characteristic function
of $F$ is given by
 $$
 \varphi_{F}(\theta) = \exp \Big\{
 \int^\infty_0\psi(e^{  -sQ }\theta)ds\Big\},
 $$
where $\psi$ is given in $(2.1)$. In particular, the generating
triplet of $P(t,x,\cdot)$ is given by
$(b_{\infty},\sigma_\infty,\rho_\infty)$, where
\begin{eqnarray*}
b_{\infty}&=&   Q^{-1} b+\int_R\int_0^\infty e^{-sQ}z\{I_{\{|e^{-sQ}z|\leq 1\}}-I_{\{|z|\leq 1\}}\}ds\rho(dz) \\
 \sigma^2_\infty&=&\int_0^\infty e^{-2sQ}\sigma^2ds  \\
 \rho_\infty (E)&=&\int_0^\infty\rho(e^{sQ}E)ds, \quad E\in
 \mathfrak{B}(R).
\end{eqnarray*}

(b) If  (2.4) fails to hold, then $X$ has no invariant distribution,
and moreover, for any $x\in R$, $P(t,x,\cdot)$ does not converge to
any probability measure as $t\rightarrow \infty$.
\end{proposition}

According to Proposition 2.2, under condition (2.4), $X$ is ergodic.
We shall use this in the section 4.

To begin, we introduce some definitions following \cite{Bertoin:96}.
 \vskip3mm
\begin{definition} {\bf (Occupation measure)}\qquad For every
$t>0$,
 the occupation measure on the time $[0,t]$ is the measure $\mu_t$
 given for every measurable function $f: R\rightarrow [0,\infty)$ by
 $$
 \int_Rf(x)\mu_t(dx)=\int_0^t f(X_s)ds.
 $$
\end{definition}

 When the occupation measure is absolutely continuous, Lebesque's
 differentiation theorem enables us to define a particular version
 of the density of the occupation measure, called the local time.

 \begin{definition}{\bf (Local time)}\qquad  For every $t\geq o$ and
 $x\in R$, the quantity
 $$
\limsup\frac{1}{2\epsilon}\int^t_0I_{\{|X_s-x|<\epsilon\}}ds
$$
denoted by $L(x,t)$ and called the local time at level $x$ and time
$t$.
\end{definition}

The local time is defined at last in three different ways, namely
via stochastic calculus, via excursion theory, and via additive
functions. Definition 2.2 is the first approach. the
Blumenthal-Getoor local time is defined as the unique continuous
additive function supported by a single point $x$, and $L(x,t)$
exist if and only if $x$ is a regular point. See \cite{geman:80} and
\cite{Kallenberg:2001}

\section{The local time of the Markov  process of O-U type  }

In this section, we will  give a completely general criterion for
the existence of local time as a density of occupation measure. The
proof is based on Fourier analysis approach due to S. M. Berman
\cite{berman:69}, \cite{berman:73}, See also \cite{Bertoin:96}.

 At first, we calculate Fourier transform of
$X_s-X_t $ for $0<t<s$.
  For every $\theta\in R $, by the Markov property,   the
  time-homogeneous and (2.2),
\begin{eqnarray}
& &\varphi_{X_s-X_t}(\theta)= Ee^{i<\theta, X_s-X_t>}
 = E\ ( E[e^{i \theta( X_s-X_t) }\big|\mathfrak{F_t} ])\nonumber\\
& =& E\Big[ e^{-i \theta X_t}E_{X_t }e^{i \theta(X_{s-t})}\Big]\nonumber\\
& =& Ee^{-i \theta X_t }\exp\Big\{i X_te^{-(s-t)Q}\theta  -
\int^{s-t}_0\psi(e^{-uQ}\theta)du\Big\}\nonumber\\
& =& Ee^{iX_t \theta (e^{-(s-t)Q}-1)}\exp\Big\{-
\int^{s-t}_0\psi(e^{-uQ}\theta)du\Big\}\nonumber\\
& =&
\exp\Big\{ixe^{-tQ}(e^{-(s-t)Q }-1)\theta -\int^t_0\psi(e^{-uQ }\theta(e^{-(s-t)Q }-1))du\nonumber\\
  & & \hspace{5cm} -\int^{s-t}_0\psi(e^{-uQ }\theta)du\Big\}.\label{4.1}
\end{eqnarray}

 Hence, we have
\begin{eqnarray}
|\varphi_{X_s-X_t}(\theta)|&=  & \exp\Big\{
-\int^t_0\mathfrak{Re}\psi(e^{-uQ }\theta(e^{-(s-t)Q }-1))du
-\int^{s-t}_0\mathfrak{Re}\psi(e^{-uQ }\theta)du\Big\}\nonumber\\
&\leq& \exp\Big\{-\int^{s-t}_0\mathfrak{Re}\psi(e^{-uQ
}\theta)du\Big\},\label{3.2}
\end{eqnarray}
where $\mathfrak{Re}\psi$ is the real part of $\psi$.

\begin{theorem}\quad Let $X$ be defined by (1.1) and $\psi$
is the characteristic exponent of
 $Z$. Suppose that either of the following conditions holds true for
 each $t\in R^+$

 (a)   $\sigma>0$.

 (b) There exist constants $\alpha\in(0,2)$ and $c>0$ such that
  \begin{eqnarray}
 \int_{\{z:|vz|\leq 1\}}|vz|^2\rho(dz)\geq c|v|^{2-\alpha}\label{1}
\end{eqnarray}
for any $v\in R^d$ satisfying $|v|\geq 1$. Then the local time exist
in $L^2( dp)$ a.e.
\end{theorem}
\begin{proof}\quad Introduce the measure $\mu$ by
\begin{eqnarray}
\int_R f(x)\mu(dx)=\int_0^\infty e^{-2Qs}f(X_s)ds=\int_0^\infty
dte^{-2Qt}\int_R f(x)\mu_t(dx),\label{3.3}
\end{eqnarray}
the occupation measure $\mu_t$ is absolutely continuous with respect
to $\mu$ with density bounded from above by $e^t$. Now, by Fubini's
theorem and Plancherel's theorem, what we have to check is
\begin{eqnarray}
\int^\infty_{-\infty}E(|\mathfrak{F}\mu(\theta)|^2)d\theta<\infty,
\label{4.2}
\end{eqnarray}
where $\mathfrak{F}\mu(\theta)$ denotes  the Fourier transform  of
$\mu$.

Noted that $E(|\mathfrak{F}\mu(\theta)|^2)$ is a non-negative real
function, from  the definition of $\mu$ (\ref{3.3}), (\ref{3.2}),
and Fubini's theorem,
\begin{eqnarray}
E(|\mathfrak{F}\mu(\theta)|^2)&=&E[\mathfrak{F}\mu(\theta)\mathfrak{F}\mu(-\theta)]\nonumber\\
&=&E[(\int_0^\infty e^{-2Qs}\exp\{i\theta X_s\}ds)(\int_0^\infty
e^{-2Qt}\exp\{-i\theta X_t\})]\nonumber\\
&=&E(\int_0^\infty  \int_0^\infty
 e^{-2Q(s+t)}\exp\{i\theta(X_s-X_t)\}dtds)\nonumber\\
&\leq& \int_0^\infty  \int_0^\infty
 e^{-2Q(s+t)}\exp\Big\{-\int^{s-t}_0\mathfrak{Re}\psi(e^{-uQ
}\theta)du\Big\} dtds).
 \label{3}
\end{eqnarray}

When $\sigma>0$, by (\ref{3.2}) and the definition of $\psi$,
\begin{eqnarray}
 \exp\Big\{-\int^{s-t}_0\mathfrak{Re}\psi(e^{-uQ
}\theta)du\Big\}
 \leq\exp\{-\frac{1}{2}(\theta\sigma)^2
\int^{s-t}_0e^{-2uQ }du\},\label{4}
\end{eqnarray}
By (\ref{3}), (\ref{4}) and Fubini's theorem,
\begin{eqnarray}
\int_{R}E(|\mathfrak{F}\mu(\theta)|^2)d\theta
&\leq&2\int_Rd\theta\int_0^\infty dte^{-2Qt}\int_t^\infty
dse^{-2Qs}\exp\{-\frac{(\theta\sigma)^2}{2 }
\int^{s-t}_0e^{-2uQ }du\}\nonumber\\
&=&2\int_Rd\theta\int_0^\infty dte^{-2Qt}\int_t^\infty
dse^{-2Qs}\exp
\{-\frac{(\sigma\theta)^2}{4Q} (1-e^{-2(s-t)Q})\}d\theta \nonumber\\
&=&2\int_0^\infty dte^{-2Qt}\int_t^\infty
dse^{-2Qs}\int_R\exp\{-\frac{(\sigma\theta)^2}{4Q} (1-e^{-2(s-t)Q})\}\nonumber\\
&=&C\int_0^\infty dte^{-2Qt}\int_t^\infty
 e^{-2Qs}  (1-e^{-2(s-t)Q})^{-\frac{1}{2}}ds\nonumber\\
&\leq&C\Gamma(\frac{1}{2})<\infty.\label{5}
\end{eqnarray}
So that, whenever $\sigma>0$, assertion (a) follows (\ref{5}).

Turing to (b), by (\ref{3.2}),
\begin{eqnarray}
 \exp\Big\{-\int^{s-t}_0\mathfrak{Re}\psi(e^{-uQ
}\theta)du\Big\}
 &\leq& \exp\Big\{-\int^{s-t}_0\int_R[1-\cos(e^{-uQ }\theta z)]\rho
(dz)du\Big\}.\label{6}
\end{eqnarray}
Let
 $$
J(\theta)=\exp\Big\{-\int^{s-t}_0\int_R[1-\cos(e^{-uQ }\theta
z)]\rho (dz)du\Big\}
$$

Using the inequality $1-\cos \geq 2(x/\pi)^2 $ for $|x|\leq \pi$ and
assumption (\ref{1}), when $e^{-sQ}\theta\geq 1$,
\begin{eqnarray}
 J(\theta)
 \leq  \exp\{-C
\int^{s-t}_0|e^{- uQ }\theta|^{2-\alpha}du\}.\label{7}
\end{eqnarray}
 By (\ref{3}) and (\ref{7}), we have
\begin{eqnarray}
&&\int_{R}E(|\mathfrak{F}\mu(\theta)|^2)d\theta\nonumber\\
&\leq&2\int_R\int_0^\infty
 \int_t^\infty J(\theta)e^{-2(s+t)Q}dsdtd\theta\nonumber\\
&=& 4 \int_0^\infty  \int_t^\infty\int^\infty_{e^{sQ}}
J(\theta)e^{-2(s+t)Q} d\theta dsdt+4\int_0^\infty
 \int_t^\infty\int_0^{e^{sQ}}
J(\theta)e^{-2(s+t)Q}dsd\theta dsdt\nonumber\\
&\leq&   4 \int_0^\infty  \int_t^\infty\int^\infty_{e^{sQ}} \exp\{-C
\int^{s-t}_0|e^{- uQ }\theta|^{2-\alpha}du\}e^{-2(s+t)Q} d\theta
dsdt
\nonumber\\
&&\mspace{100mu} +4\int_0^\infty
dte^{-2Qt}\int_t^\infty\int_0^{e^{sQ}}
J(\theta)e^{-2(s+t)Q}d\theta dsdt\nonumber\\
&\leq& C\Gamma(\frac{1}{2-\alpha})+4\int_0^\infty
 \int_t^\infty\int_0^{e^{sQ}}  e^{-2(s+t)Q} d\theta dsdt\label{3.11}\\
&\leq& \infty.\nonumber
\end{eqnarray}
The proof is complete.
\end{proof}

\begin{remark}\qquad The local time could be expressed at the "sum
of times spent at $x$ up to time $t$". To avoid fixed $t$, $\mu$ is
defined. For the $2Q$ in $e^{-2Q}$ of $\mu$, it is used in
(\ref{3.11}).
\end{remark}
\begin{remark}\qquad  As above, Theorem 3.1 is not true $a.s.$ at every $x$.
From the point of view of occupation densities, such aberrant
behavior at a single state is irrelevant. the local time as
occupation density is different from the Blumenthal-Getoor local
time, as a continuous additive function of some point.  For example,
when $Z$ is a Poisson process, the start point is 0, then there is
the  Blumenthal-Getoor local time at 0 about $X$. In fact, 0 is a
holding point, so a regular point. But there do not exit a
occupation density about $X$.
\end{remark}

\begin{remark}\qquad Meyer \cite{me:75} has proved:
 let the process $Y=(Y_t)$, adapted to the natural $\sigma-$fields
 of a Brownian motion $W=(W_t)$, have trajectories of bounded
 variation; then there exit an occupation density of $W_t+Y_t$. Theorem 3.1
 asserts that there are local time when $W$ is a general O-U
 process, and $Y$ is O-U type of pure jump.
\end{remark}

\section{Some properties of local time}

In this section, we shall obtain the smoothness of the local time in
the time variable,  when the level has been fixed. At the end of the
section, we shall discuss the limit property of the local time at
$t\rightarrow \infty$ when $X$ is ergodic.

We assume that the conditions of Theorem 3.1 is satisfied in this
section. Lebesgue's differentiation theorem enables us to define a
particular version of the occupation density, called the local time
as
$$
\limsup\frac{1}{2\epsilon}\int^t_0I_{\{|X_s-x|<\epsilon\}}ds
$$
for every $t\leq0$ and $x\in R$. Now, we can replace "$\limsup$" by
"$\lim$" in the definition of local time. Before proving it, we have
some lemmas. The following Lemma come from Masuda \cite{ma:04}.

 \begin{lemma}\quad The following statements hold true  for each $t\in
 R_+$.

  (a) If $\sigma>0$, then $P(t,x,\cdot)$ admits a $C_b^\infty$
  density.

  (b) If there exist constants $\alpha\in (0,2)$ and $c>0$ such that
  (\ref{1}) satisfy, then $P(t,x,\cdot)$ admits a $C_b^\infty$
  density.
\end{lemma}
 For L\'{e}vy process $Z$, because it has stationary independent
 increments, $p(t,x,y)=p(t,0,y-x)$ for every $t\in R_+$ and $x\in
 R$. Unfortunately, there is not this property for $X$, but there is
 a similar property as following:

\begin{lemma}\quad Let $p(t,x,y)$ be the density of
$P(t,x,\cdot)$, then
\begin{eqnarray}
p(t,x,y)=p(t,0,y-xe^{-tQ}).\label{4.1}
\end{eqnarray}
Moreover, $p(t,x,y)$ is continuous about $x$ and tends to 0 as
$x\rightarrow \infty$.
\end{lemma}
\begin{proof}\quad
This is immediate from   the inversion formula and (2.2). The last
assertion stems from Lemma 4.1 and the property of probability
density function.
\end{proof}

\begin{theorem}\quad For a.e. $x\in R$
$$
\lim_{\epsilon\rightarrow
0+}\frac{1}{2\epsilon}\int_0^tI_{\{|X_s-x|<\epsilon\}}ds=L(x,t)
$$
uniformly on compact intervals of time, in $L^2(P)$. As a
consequence, the process $L(x,\cdot)$ is continuous a.s.
\end{theorem}
 \begin{proof}\quad By Theorem 3.1, there exists a local time
 $L(x,\tau)$ in $L^2(dy\otimes dP)$, where $\tau$ is an independent
 random time with an exponential distribution of parameter 1.
 Mimicking the argument of Bertoin  \cite{Bertoin:96}, for $a.e.
 y\in R$, the following convergence holds in $L^2(P)$:
\begin{eqnarray}
\lim_{\epsilon\rightarrow 0+}\frac{1}{2\epsilon}\int_0^\tau
I_{\{|X_s-y|<\epsilon\}}ds=\lim_{\epsilon\rightarrow
0+}\frac{1}{2\epsilon}\int_{y-\epsilon}^{y+\epsilon}L(v,\tau)dv=L(y,\tau).\label{4.2}
\end{eqnarray}
Pick any $y$ for which (\ref{4.2}) is fulfilled and for every
$\epsilon>0$, consider the martingale
\begin{eqnarray}
M_t^\epsilon=E(\frac{1}{2\epsilon}\int_0^\tau
I_{\{|X_s-y|<\epsilon\}}ds|\mathfrak{F'_t}), \mspace{20mu}  t\leq 0,
\label{4.3}
\end{eqnarray}
where $\mathfrak{F'_t}=\mathfrak{F_t}\vee\sigma(t\wedge\tau)$. By
(\ref{4.2}), and Doob's maximal inequality, $M_t^\epsilon$ converges
as $\epsilon\rightarrow 0+$, uniformly on $t\in [0,\infty)$,  in
$L^2(P)$.

By the Markov property and the lack of memory of the exponential
law, we have $a.s.$
\begin{eqnarray}
M_t^\epsilon= \frac{1}{2\epsilon}\int_0^{t\wedge\tau}
I_{\{|X_s-y|<\epsilon\}}ds +I_{\{t<\tau\}}f_\epsilon(X_t),
\label{4.4}
\end{eqnarray}
where
$$
f_\epsilon(x)=E_x(\frac{1}{2\epsilon}\int_0^\tau
I_{\{|X_s-y|<\epsilon\}}ds).
$$

Now, what we have to do is proving $f_\epsilon(X_t)$ convergence
uniformly on $t\in [0,\infty)$. Applying Fubini's theorem,

$$
f_\epsilon(x)= \frac{1}{2\epsilon}\int_0^\infty
e^{-t}P_x(|X_t-y|<\epsilon)dt).
$$
Applying  Lemma 4.2, we get our assertion.
\end{proof}

By  Theorem 4.1, $a.e.$ $x\in R$, $L(x,t)$ is continuous additive
function about $t\in R+$,   $L(x,t)$ also is the Blumenthal-Getoor
local time. Hence we have the following corollary:

\begin{corollary}
\quad Under the conditions of Theorem 3.1, $a.e.$  x in the range of
$X$  are regular.
\end{corollary}

Recalling Proposition 2.2, assume that
\begin{eqnarray}
\int_{|z|>1}\log|z|\rho(dz)< \infty,\label{4.5}
\end{eqnarray}
 then there
exists a limit distribution $F$ such that
$$
P(t,x,A)\rightarrow F(A) \mspace{20mu}\mbox{as}\ \  t\rightarrow
\infty
$$
for any $x\in R$ and Borol set $A$. This $F$ is the unique invariant
distribution of $X$. Hence under (\ref{4.5}), $X$ is ergodic. By the
ergodic theorem, we have
\begin{eqnarray}
\lim_{t\rightarrow
\infty}\frac{1}{t}\int_0^tI_{\{|X_s-x|<\epsilon\}}ds=\mu_F(B(x,\epsilon
)), \qquad\mbox{in}\qquad L^2(P).\label{4.6}
\end{eqnarray}

If the conditions of Theorem 3.1 is holding, $F$ has a density $f$
by Lemma 4.1, hence,

\begin{eqnarray}
\lim_{\epsilon\rightarrow 0}\lim_{t\rightarrow
\infty}\frac{1}{2t\epsilon}\int_0^tI_{\{|X_s-x|<\epsilon\}}ds=\lim_{\epsilon
\rightarrow 0}\frac{\mu_F(B(x,\epsilon
))}{2\epsilon}=f(x).\label{4.6}
\end{eqnarray}

On the other hand, by Theorem 4.1,

\begin{eqnarray}
\lim_{t\rightarrow \infty}\lim_{\epsilon\rightarrow
0}\frac{1}{2t\epsilon}\int_0^tI_{\{|X_s-x|<\epsilon\}}ds=\lim_{t
\rightarrow \infty}\frac{L(x,t)}{t}.\label{4.7}
\end{eqnarray}
We can get
$$
\lim_{t \rightarrow \infty}\frac{L(x,t)}{t}=f(x),
\qquad\mbox{in}\qquad L^2(P),
$$
if the limits in (\ref{4.6}) can commute. But this is obvious,
because $$ \lim_{\epsilon\rightarrow
0+}\frac{1}{2\epsilon}\int_0^tI_{\{|X_s-x|<\epsilon\}}ds=L(x,t) $$
uniformly on $[0,t]$ for any $t\in R^+$. More precisely, we have the
following:
\begin{theorem}
Assume that the conditions of Theorem 3.1 and (\ref{4.5}) hold true,
then $a.e.$ $x$ in the range of $X$,

$$
\lim_{t \rightarrow \infty}\frac{L(x,t)}{t}=f(x),
\qquad\mbox{in}\qquad L^2(P).
$$

\end{theorem}

\section{The first passage cross a lever}

Let  $X=\{X_t,t\in R^+,P^{x}\} $ be   one dimensional  Markov
process of O-U type defined by $(1.1)$ taking values in $R$. Given a
real number $a>x$, let us introduce the first passage time strictly
above $a$, $ T_a=\inf \{t\geq 0: X_t> a\} $, and let
$\sigma_a=\inf\{t\geq 0: X_t=a\}$ provided that the sets in braces
is not empty,  and $ +\infty$ otherwise.

 When $Z$ is a  L\'{e}vy process with non-positive jumps,
 $ \Delta X_t=\Delta Z_t\leq 0$. If $T_a<\infty$, one gets
 immediately
 $$
X_{T_a}=a.\eqno{(3.1)}
$$
  Using martingle technique, Hadjiev \cite{Hadj:85}
 proved that
 $$
 E\exp \{-\theta T_a \}=\frac{\int_0^\infty y^{\theta/Q-1}\exp \{xy+g(y)\}dy}
 {\int_0^\infty y^{\theta/Q-1}\exp \{ay+g(y)\}dy}, \ \
 \theta>0,\eqno{(3.2)}
 $$
where
$$
g(y)=Q^{-1}\int_1^yu^{-1}\psi(iu)du,\ \ y>0.
$$

 When $Z$ is a  L\'{e}vy process with positive jumps, does the similar
 property (3.1)
 hold? We will prove that the answer is negative.

\begin{lemma}\qquad Let $X=\{X_t,t\in R^+,P^{x}\} $ be
a Markov process of O-U type defined by $(1.1)$. Then for every
$x\neq 0$ and $y\in R^d$, the potential measure of $X$ is diffuse,
that is,
$$
U(x, \{y\})=0.
$$
\end{lemma}

{\bf Proof.}\qquad Since
 $$
 X_t = e^{-tQ}x + \int^t_0e^{(s-t)Q}dZ_t
 $$
 and the distribution of $Z$ is a diffuse except when $Z$ is a compound
  Poisson process     for every
$x\neq 0  $,
$$
P^{x}\{X_t=y\}=0 ,
$$
which implies
$$
U(x,\{y\})=\int_0^\infty P^x\{X_t=y\}dt=0.
$$

\begin{theorem}\qquad Let $X=\{X_t,t\in R^+,P^{x}\} $ be a
Markov process of O-U type defined by $(1.1)$. If
$\rho(-\infty,0)=0$, we have
$$
P^x\{X_{T(a)-}<a=X_{T(a)}\}=0.
$$
\end{theorem}
\vskip3mm {\bf Proof.}\qquad Let$f,g\geq 0$ be two Borel functions
with $f(a)=0$.
  Applying the compensation formula and recalling
$\Delta X_t=\Delta Z_t\geq 0$, we have
\begin{eqnarray}
& &\int_{0\leq y< a\leq z}f(y)g(z)P^x\{X_{T(a)-}\in dy,X_{T(a)}\in dz\} \nonumber\\
&=&E^x\big(f(X_{T(a)-})g(X_{T(a)})\big)\nonumber\\
&=& E^x\big(\sum_{t\geq 0}f(X_{t-})g(X_{t-}+\Delta X_t)
I_{\{a- \Delta X_t\leq X_{t-}<a\}}\big)\nonumber\\
&=&\int_0^\infty dt E^x\big(f(X_{t-})I_{\{X_{t-}<
a\}}\int^\infty_0g(X_{t-}+s)I_{\{s\geq
a-X_{t-}\}}\rho(ds)\big)\nonumber \\
&=&\int_{0\leq y< a\leq z}f(y)g(z)P^x\{X_{T(a)-}\in dy,X_{T(a)}\in dz\} \nonumber\\
 & =& \int_0^\infty
dt\int_{0\leq y< a,s\geq a-y}f(y)g(y+s)P^x\{X_{t}\in
dy\}\rho(d s)\nonumber\\
& =&\int_{0\leq y< a\leq z}f(y)g(z)U(x,dy)\rho(dz-y).\nonumber
\end{eqnarray}
Taking $f=I_{[0,a)}$ and $g=I_{\{a\}}$, we obtain
$$
P^x\{X_{T(a)-}<a=X_{T(a)}\}=\int_{[0,
a)}U(x,dy)\rho(\{a-y\}).\eqno{(3.3)}
$$
There are at most countably many $y\in [0,a)$ with $\rho(\{a-y\})>0$
and $\rho(\{0\})=0$. Moreover the potential measure is diffuse by
Lemma 3.1. Hence the right-hand side of (3.3) is zero. \vskip3mm

We deduce from Theorem 3.1 that $X$ is a.s. continuous at time
$T(a)$ on the event $X_{T(a)}=a$, so $P\{X_{T(a)-}=a|X_{T(a)}=a\}=1$
on $P\{X_{T(a)}=a\}>0$ and  $P\{X_{T(a)}>a|X_{T(a)-}<a\}=1$ on
$P\{X_{T(a)-}<a\}>0$.

It is well known from \cite{Bertoin:96} that we can write $Z_t=
at+\sigma W_t+ Z^1_s$, where $at$ is   a drift, $W_t$ is the
Brownian motion and $Z^1_s$ is a L\'{e}vy process of pure jumps
type. Hence $X$ has the following decomposition:
$$
X_t=e^{-tQ}x + a\int^t_0e^{(s-t)Q}ds+ \sigma\int^t_0e^{(s-t)Q}dW_s
+\int^t_0e^{(s-t)Q}dZ^1_s.\eqno{(3.4)}
$$
 \begin{theorem}\quad Assume that
$x=a=\sigma=0$ in (3.4), if
$$
\rho(-\infty, 0)=0 \mbox{\ and\ }\int_0^1x\rho(dx)=C<+\infty,
\eqno{(3.5)}
$$
then $X_{T(a)}>a$  a.s.
\end{theorem}
\vskip3mm {\bf Proof.}\quad Note that $0<e^{(s-t)Q}<1$ for $s<t$,
$X$ gets its supremum just by jumping, that is, $P\{X_{T(a)-}<a\}>0$
by Theorem 3.1, the theorem is proved.

\end{document}